\documentclass[a4paper,12pt, twoside, reqno]{amsart}
\usepackage{amsthm}
\usepackage{amsmath}
\usepackage{amssymb}
\usepackage{eucal}

\newtheorem{thm}{Theorem}[section]

\newtheorem{prop}[thm]{Proposition}

\theoremstyle{definition}
\newtheorem{rmk}[thm]{Remark}

\catcode`@=11
\catcode`@=12

\let\emptyset=\varnothing

\begin{document}

\title{Multiple fixed point theorems for contractive and Meir-Keeler type mappings defined on partially ordered spaces with a distance}
\author{Mitrofan M. Choban$^{1}$ and Vasile Berinde$^{2,3}$}

\medskip
\begin{abstract}
We introduce and study a general concept of multiple fixed point for mappings defined on partially ordered distance spaces in the presence of a contraction type condition and appropriate monotonicity properties. This notion and the obtained results complement the corresponding ones from [Choban, M., Berinde, V., {\it A general concept of multiple fixed point for mappings defined on  spaces with a distance} (submitted)] and also simplifies some concepts of multiple fixed point considered by various authors in the last decade or so. 
\end{abstract}

\maketitle

\pagestyle{myheadings} \markboth{Mitrofan M. Choban and Vasile Berinde} {Multiple fixed point theorems for contractive and Meir-Keeler}

 \section{Introduction}

In a previous paper \cite{ChoB}, the authors have introduced and studied a general concept of multidimensional fixed point for mappings defined on a distance space and satisfying a certain contraction condition. 

Several interesting new results that generalise, extend and unify corresponding related results from literature for the case of non ordered distance spaces were obtained. However, the great majority of the multidimensional fixed point theorems existing in literature were established in the setting of a partially ordered metric space or of a partially ordered generalised metric space. Therefore, the main aim of this paper is to study the concept of multidimensional fixed point introduced in  \cite{ChoB} for the case of mappings defined on partially ordered distance space, thus extending and complementing most of the results established in \cite{ChoB}.

We start by presenting a brief survey on the notion of {\it multidimensional fixed point}, which  naturally emerged from the rich literature produced in the last four decades devoted to  coupled fixed points. The concept itself of {\it coupled fixed point}  has been first introduced and studied by V. I. Opoitsev,  in a series of papers he published in the period 1975-1986, see \cite{Op75}-\cite{Op86}. Opoitsev has been inspired by some concrete problems arising in the dynamics of collective behaviour in mathematical economics and considered the coupled fixed point problem for mixed monotone nonlinear operators which also satisfy a nonexpansive type condition.  
 
In 1987, Guo and Lakshmikantham  \cite{Guo},  apparently not being aware of Opoitsev's previous results  \cite{Op75}-\cite{Op86}, have studied coupled fixed points in connection with coupled quasi-solutions of an initial value problem for ordinary differential equations. Amongst the subsequent developments we quote the following works: \cite{Guo88}; \cite{Chen}, containing coupled fixed point results of $\frac{1}{2}$-$\alpha$-condensing and mixed monotone operators, where $\alpha$ denotes the Kuratowski's measure of non compactness, thus extending some previous results from \cite{Guo} and \cite{Shen};  \cite{Chang91}, which discusses some existence results and iterative approximation of coupled fixed points for mixed monotone condensing set-valued operators;  \cite{Chang96} where the authors obtained coupled fixed point results of $\frac{1}{2}$-$\alpha$-contractive and generalized condensing mixed monotone operators.

More recently,  Gnana Bhaskar and Lakshmikantham in \cite{Bha} established coupled fixed point results for mixed monotone operators in partially ordered metric spaces in the presence of a Bancah contraction type condition. Essentialy, the results by Bhaskar and Lakshmikantham in \cite{Bha}  combined, in the context of bivariable mixed monotone mappings, the main fixed point results previously obtained by Nieto and Rodriguez-Lopez  \cite{Nie06} and  \cite{Nie07}, for the case of one variable increasing and decreasing nonlinear operator, respectively.  The last two papers are, in turn, a continuation of the hybrid fixed point theorem established in the seminal paper of Ran and Reurings \cite{Ran}, which has the merit to combine a metrical fixed point theorem (the contraction mapping principle) and an order theoretic fixed point result (Tarski's fixed point theorem). 

Various applications of the theoretical results in coupled fixed point theory were also considered, for the case of: a) Uryson integral equations \cite{Op78}; b) a system of Volterra integral equations  \cite{Chen}, \cite{Chang96}; c) a class of functional equations arising in dynamic programming \cite{Chang91}; d) initial value problems for first order differential equations with discontinuous right hand side \cite{Guo}; e) (two point) periodic boundary value problems \cite{Ber12a}, \cite{Bha}, \cite{Cir}, \cite{Urs}; f) integral equations and systems of integral equations \cite{Agha}, \cite{Algha}, \cite{Aydi}, \cite{BS}, \cite{Gu}, \cite{Hus}, \cite{Shat}, \cite{Sint}, \cite{Xiao}; g) nonlinear elliptic problems and delayed hematopoesis models \cite{Wu}; h) nonlinear Hammerstein integral equations \cite{Sang}; i) nonlinear matrix and nonlinear quadratic equations \cite{AghaFPT}, \cite{BS}; j) initial value problems for ODE \cite{Amini}, \cite{Samet} etc.
 
 For a very recent account on the developments of coupled fixed point theory, we also refer to \cite{JNCA}.
 
On the other hand, in 2010, Samet and Vetro  \cite{SV} apart of some coupled fixed point results they have established, considered a concept of {\it fixed point of m-order} as a natural extension
of the notion of coupled fixed point.  Then, in 2011, mainly inspired by \cite{Bha}, Berinde and Borcut \cite{BB1} introduced the concept of {\it triple fixed
point} and proved existence and existence and uniqueness triple fixed point theorems for three-variable mixed monotone mappings, while, in 2012, Karapinar and Berinde \cite{KB}, have studied quadruple fixed points of nonlinear contractions in partially ordered metric spaces. 

After these starting papers, a substantial number of articles were dedicated to the study of triple fixed points, quadruple fixed points, as well as to multiple fixed points (also called  {\it fixed point
of m-order}, or  "a multidimensional fixed point", or "an $m$-tuplet
fixed point", or "an $m$-tuple fixed point"), see \cite{Aga14}, \cite{Aga}, \cite{Al}, \cite{Kar13}, \cite{Kar13a}, \cite{Lee}, \cite{Mutlu}, \cite{Ola}, \cite{R1}-\cite{Rus}, \cite{Sol}, \cite{Wang}, \cite{ZCS}, which form a very selective list contributions.

Starting from this background, the main aim of the present paper is to study the concept of multidimensional fixed point introduced in \cite{ChoB} but for mappings defined on partially ordered distance space, in the presence of a contraction type condition and appropriate monotonicity properties, thus extending and complementing the results established in \cite{ChoB}.

This approach is based on the idea to reduce the study of multidimensional  fixed points and coincidence points to 
the study of usual one-dimensional fixed points for an associate operator. Note that, the first author who reduced the problem of finding a coupled fixed point of mixed monotone operators to the problem of finding a fixed point of an increasing one variable operator was Opoitsev, see for example  \cite{Op78}.

 \section{Preliminaries }

By a space we understand a  topological  $T_0$-space.
We use the terminology from \cite{Eng, GD, RP, C1}.

Let $X$ be a non-empty set and $d : X\times X \rightarrow \mathbb R$ be a mapping such that:

($i_m$) $d(x, y) \geq  0$, for all $x, y \in X$;

($ii_m$) $d(x, y) + d(y,x) = 0$ if and only if $x = y$.

Then $d$ is called a {\it distance} on $X$, while  $(X, d)$ is called a {\it distance space}.

Let $d$ be a distance on $X$ and $B(x,d,r)$ = $\{y \in X: d(x, y) < r\}$ be the {\it ball} with the center $x$
and radius $r > 0$. The set $U \subset X$ is called {\it $d$-open} if for any $x \in U$ there exists $r > 0$
such that $B(x,d,r) \subset U$. The family $\mathcal T(d)$ of all $d$-open subsets is the topology on $X$ 
generated by $d$. A distance space is a {\it sequential space}, i.e., a space for which a set $B \subseteq X$ is closed if and only 
if together with any sequence it contains all its limits \cite{Eng}.
   
  Let $(X, d)$ be a  distance space, $\{x_n\}_{n \in \mathbb N}$ be a sequence in $X$ and $x \in X$. We say that the sequence   $\{x_n\}_{n \in \mathbb N}$ is:

1)  {\it convergent to} $x$ if and only if $\lim_{n\rightarrow \infty }d(x, x_n) = 0$. 
We denote this by $x_n\rightarrow x$ or $x = \lim_{n\rightarrow \infty }x_n$ (really, we may denote   $x \in \lim_{n\rightarrow \infty }x_n$);

2)   {\it convergent} if  it converges to some point $x$ in  $X$;

3)   {\it Cauchy} or {\it fundamental} if $\lim_{n, m\rightarrow \infty }d(x_n, x_m) = 0$.
    
  A  distance space $(X, d)$ is called {\it complete} if  every
Cauchy sequence in $X$ converges to some point $x$  in $X$.

    Let $X$ be a non-empty set and $d $ be a distance on $X$. Then:  
\begin{itemize}
\item $(X, d)$ is called a {\it symmetric space} and $d$ is called a {\it symmetric} on $X$ if
  
($iii_m$) $d(x, y)$ = $d(y, x)$, for all $x, y \in X$;
\item  $(X, d)$ is called a {\it quasimetric space} and $d$ is called a {\it quasimetric} on $X$  if

($iv_m$) $d(x, z) \leq d(x, y) + d(y, z)$, for all $x, y, z \in X$;
\item  $(X, d)$ is called a {\it metric space} and $d$ is called a {\it metric} if  $d$ is a symmetric and a  quasimetric, simultaneously.
\end{itemize}

      Let $X$ be a non-empty set and  $d(x, y) $ be a distance on $X$
   with the following property:

($N$) for each point $x \in X$ and any $\varepsilon > 0$ there exists $\delta = \delta (x,\varepsilon ) > 0$
such that from $d(x, y) \leq \delta $ and $d(y,z) \leq \delta $ it follows $d(x, z) \leq \varepsilon $.

\noindent
Then $(X, d)$ is called an {\it N-distance space} and $d$ is called an {\it N-distance}
on $X$.
If $d$ is a symmetric, then we say that $d$ is an $N$-symmetric.

Spaces with $N$-distances were studied by V. Niemyzki  \cite{Ne1} and by S. I. Nedev \cite{N}. If $d$ satisfies the condition 

($F$) for any $\varepsilon > 0$ there exists $\delta = \delta (\varepsilon ) > 0$
such that from $d(x, y) \leq \delta $ and $d(y,z) \leq \delta $ it follows $d(x, z) \leq \varepsilon $,

then  $d$ is called an {\it F-distance}  or a {\it Fr\'echet distance}  and  $(X, d)$ is 
called an {\it F-distance space}. Any $F$-distance $d$ is an $N$-distance, too.
If $d$ is a symmetric and an $F$-distance on a space $X$, then we say that $d$ is an $F$-symmetric.

 \begin{rmk}\label{rm1.1}  If $(X, d)$ is an $F$-symmetric space, then any convergent sequence is a 
Cauchy sequence. For $N$-symmetric spaces and for quasimetric spaces this assertion is not more true.
 \end{rmk}

 If $s > 0$ and $d(x,y) \leq s [d(x,z) + d(z, y)]$ for all points $x, y, z \in X$, then we say that $d$ is an {\it $s$-distance}.
 Any $s$-distance is an $F$-distance.

 A distance space $(X, d)$ is called an {\it H-distance space} if, for any two distinct points $x, y \in X$,
 there exists $\delta  = \delta (x, y) > 0$ such that $B(x, d, \delta ) \cap B(y, d, \delta ) = \emptyset $.
 
 \begin{rmk}\label{rm1.2} A distance space $(X, d)$  is an $H$-distance space if and
only if any convergent sequence  in $X$ has a unique limit point. 
 \end{rmk}

\medskip

 \section{Ordering on Cartesian product of distance spaces}

Let $(X,d)$ be a distance space, $m \in \mathbb N $ = $\{1, 2, ...\}$. On $X^m$ consider the distances
$$d^m((x_1,..., x_m), (y_1,...,y_m)) = sup \{d(x_i,y_i): i \leq m\}$$ 
and 
$$\bar{d}^m((x_1,..., x_m), (y_1,...,y_m)) = \sum_{i=1}^{m}  d(x_i,y_i).$$
Obviously, $(X^m, d^m)$ and $(X^m, \bar{d}^m)$ are distance spaces, too.

\begin{prop}\label{P2.1} (\cite{ChoB})   Let $(X, d)$ be a distance space. Then:

1. If $d$ is a symmetric, then $(X^m, d^m)$ and $(X^m, \bar{d}^m)$ are  symmetric spaces, too.

2. If $d$ is a quasimetric, then  $(X^m, d^m)$ and $(X^m, \bar{d^m})$ are  quasimetric spaces, too.

3.  If $d$ is a metric, then   $(X^m, d^m)$ and $(X^m, \bar{d}^m)$ are  metric spaces, too.

4.  If $d$ is an $F$-distance space, then  $(X^m, d^m)$ and $(X^m, \bar{d}^m)$ are   $F$-distance spaces, too.

5.   If $d$ is an $N$-distance space, then $(X^m, d^m)$ and $(X^m, \bar{d}^m)$ are   $N$-distance spaces, too.

6.  If $d$ is an $H$-distance space, then $(X^m, d^m)$ and $(X^m, \bar{d}^m)$ are   $H$-distance spaces, too.

7. If  $(X, d)$ is a $C$-distance space, then   $(X^m, d^m)$ and $(X^m, \bar{d}^m)$ are   $C$-distance spaces, too.

8.  If $(X, d)$ is a complete distance space, then $(X^m, d^m)$ and $(X^m, \bar{d}^m)$ are   complete distance spaces, too.

9.  If $d$ is an $s$-distance space, then $(X^m, d^m)$ and $(X^m, \bar{d}^m)$ are   $s$-distance spaces, too.

10. The spaces  $(X^m, d^m)$ and $(X^m, \bar{d}^m)$ share the same convergent sequences and the same Cauchy sequences.
Moreover, the distances $d^m$ and $\bar{d^m}$ are uniformly equivalent, i.e., for each $\varepsilon  > 0$, there exists $\delta $
= $\delta (\varepsilon ) > 0$ such that:

- from $d^m(x,y) \leq \delta $ it follows $\bar{d^m}(x,y) \leq \varepsilon $;

-  from $\bar{d^m}(x,y) \leq \delta $ it follows $d^m(x,y) \leq \varepsilon $. 
\end{prop}    
 
 Let $\preceq $ be a (partial) order on a distance space $(X, d)$.  A sequence $\{x_n\}_{n \in \mathbb N}$ is called:
\begin{itemize}
\item {\it non-decreasing} if $x_n \preceq x_{n+1}$ for each $n \in \mathbb N$;
\item {\it non-increasing} if $x_n \succeq  x_{n+1}$ for each $n \in \mathbb N$;
\item {\it monotone} if it is either non-decreasing or non-increasing.
\end{itemize}

If $(X, d, \preceq )$ is an ordered distance space and $g: X \rightarrow X$ is a mapping, then $(X, d, \preceq )$  is said to have
the sequential $g$-monotone property \cite{Bha, R2} if it verifies: 

i) if $\{x_n\}_{n \in \mathbb N}$ is a non-decreasing sequence and $\lim\limits_{n\rightarrow \infty} d(x,x_n)$ = $0$, 
then $g(x_n) \preceq g(x)$  for all $n \in \mathbb N$;

ii)  if $\{x_n\}_{n \in \mathbb N}$ is a  non-increasing sequence and $\lim\limits_{n\rightarrow \infty }d(x,x_n)$ = $0$, 
then $g(x_n) \preceq g(x)$  for all $n \in \mathbb N$.

An ordered distance space $(X, d, \preceq )$ is called {\it monotonically complete} if any monotone Cauchy sequence 
 in $X$ converges to some point   in $X$.

Fix $m \in \mathbb N$ and a subset $L \subseteq \{1, 2, ..., m\}$. Like in \cite{R1, R2, R3}, we introduce on $X$ the ordering $\preceq _L$:
$ (x_1,..., x_m) \preceq _L (y_1,...,y_m)$ iff $x_i \preceq y_i$ for $i \in L$ and $y_j \preceq x_j$ for $j \notin L$.

By construction, $(X^m, d^m, \preceq _L)$ and   $(X^m, \bar{d^m}, \preceq _L)$ are ordered distance spaces.

If $L \subseteq \{1, 2, ..., m\}$ and $M$ = $\{1, 2, ..., m\}\setminus L$, then $x \preceq _L y$ if and only if $y \preceq _M x$
for $x, y \in X^m$. Hence $\preceq _M$ is the dual (inverse) order of the order $\preceq _L$.

 \begin{prop}\label{P2.2}    Let $(X, d, \preceq )$ be a  monotonically complete distance space. Then
  $(X^m, d^m, \preceq _L)$ and   $(X^m, \bar{d^m}, \preceq _L)$ are ordered  monotonically   complete distance spaces, too.
\end{prop}    

\begin{proof}  It is obvious.
 \end{proof}

\medskip

 \section{Multiple fixed point principles for monotone type operators}

Fix $m \in \mathbb N$. Denote by $\lambda $ = $(\lambda _1,... , \lambda _m)$ a collection of mappings
$\{\lambda _i: \{1, 2, ..., m\}$ $\longrightarrow \{1, 2, ..., m\}: 1\leq i \leq  m\}$.

Let $(X, d)$ be a distance space and $F: X^m \longrightarrow X$ be an operator. The operator $F$ and the
mappings $\lambda $ generate the operator $\lambda F: X^m \longrightarrow X^m$, where 
$$\lambda F(x_1,....,x_m)
= (y_1,...,y_m) \textnormal{ and } y_i= F(x_{\lambda _i(1)}, ..., x_{\lambda _i(m)}),
$$
 for each point $(x_1, ..., x_m) \in X^m$
and any index $i \in \{1, 2, ..., m\}$. 

A point $a$ = $(a_1, ..., a_m) \in X^m$ is called a $\lambda ${\it -multiple fixed point} of the operator $F$ if $a$ = $\lambda F(a)$,
i.e., $a_i$ = $F(a_{\lambda _i(1)}, ..., a_{\lambda _i(m)})$ for any  $i \in \{1, 2, ..., m\}$.

 Let $(X, d, \preceq )$ be a (partially) ordered distance space,  $m \in \mathbb N$, $L \subseteq \{1, 2,..., m\}$
 and $F: X^m \longrightarrow X$ be an operator. In this context, we consider the following sets of assumptions that include a symmetric type contractive condition, similar to the symmetric contraction introduced and used by the second author in \cite{Ber11}.
 
 {\bf Conditions $\Omega _1$}: \\ 
 \indent 1.  $(X,\preceq )$ is a lattice;
 
 2.  If $x, y, z\in X$ and $x \preceq  y \preceq  z$, then $d(x,y) + d(y, x) \leq d(x,z) + d(z, x)$;

3. If $x, y \in X_m$, $x \not= y$ and $x \preceq _Ly$, then $\lambda F(x) \preceq _L \lambda F(y)$
 and $d^m(\lambda F(x), \lambda F(y)) $ + $d^m(\lambda F(y), \lambda F(x)) $ $< d^m(x,y)$ + $d^m(y,x)$.
 
 {\bf Conditions $\Omega _2$}:\\
\indent 1.  $(X,\preceq )$ is a lattice;
 
 2. If $x, y, z\in X$ and $x \preceq  y \preceq  z$, then $d(x,y) + d(y, x) \leq d(x,z) + d(z, x)$;

3. If $x, y \in X_m$, $x \not= y$ and $x \preceq _L y$, then $\lambda F(y) \preceq _L \lambda F(x)$
 and $d^m(\lambda F(x), \lambda F(y)) $ + $d^m(\lambda F(y), \lambda F(x)) $ $< d^m(x,y)$ + $d^m(y,x)$.

  {\bf Conditions $\Omega _3$}: \\
\indent   1.  $(X,\preceq )$ is a lattice;
 
 2.   If $x, y, z\in X$ and $x \preceq  y \preceq  z$, then $d(x,y) + d(y, x) \leq d(x,z) + d(z, x)$;

3. For any $i \in \{1,2,\dots,m\}$ the mapping $\lambda _i$ is a surjection
or, more generally, $|\cup \{\lambda _i^{-1}(j): 1\leq j \leq m\}| = m$, for each $i \in \{1,2,\dots,m\}$;

4. If $x, y \in X_m$, $x \not= y$ and $x \preceq _Ly$, then $\lambda F(x) \preceq _L \lambda F(y)$
 and $\bar{d}^m(\lambda F(x), \lambda F(y)) $ + $\bar{d}^m(\lambda F(y), \lambda F(x)) $ $< \bar{d}^m(x,y)$ + $\bar{d}^m(y,x)$.
 
  {\bf Conditions $\Omega _4$}:\\
  \indent  1.  $(X,\preceq )$ is a lattice;
 
 2.   If $x, y, z\in X$ and $x \preceq  y \preceq  z$, then $d(x,y) + d(y, x) \leq d(x,z) + d(z, x)$;

3. For any $i \in \{1,2,\dots,m\}$ the mapping $\lambda _i$ is a surjection
or, more generally, $|\cup \{\lambda _i^{-1}(j): 1\leq j \leq m\}| = m$, for each $i \in \{1,2,\dots,m\}$;

4. If $x, y \in X_m$, $x \not= y$ and $x \preceq _Ly$, then $\lambda F(y) \preceq _L \lambda F(x)$
 and $\bar{d}^m(\lambda F(x), \lambda F(y)) $ + $\bar{d}^m(\lambda F(y), \lambda F(x)) $ $< \bar{d}^m(x,y)$ + $\bar{d}^m(y,x)$.

Now we can state concisely the following general and comprehensive multidimensional fixed point result.

\begin{thm}\label{T6.1}     Let  $a \in X^m$ be a multidimensional fixed point of the operator of $F: X^m \longrightarrow X$.
If any of the Conditions $\Omega _i$, $i \in\{1, 2, 3, 4\}$, is satisfied, then the operator $F$ has a unique multidimensional fixed point.
 \end{thm}    

\begin{proof} Obviously, $(X^m, \preceq _L)$ is a lattice, too.
 Let $\rho $ = $d^m$, for $i \in \{1, 2\}$, and  $\rho $ = $\bar{d}^m$, for $i \in \{3, 4\}$.
Then for $x, y, z\in X^m$ and $x \preceq _L y \preceq _L z$ we have $\rho (x,y)+ \rho (y,x)\leq \rho (x,z) + \rho (z,x)$.
Assume that $b\in X^m$ is a  multidimensional fixed point of the operator $F$ with $b \not= a$.

{\bf Case 1}. The points $a$ and $b$ are comparable.

Assume that $a \preceq _L b$. Then $\rho (a,b)$ + $\rho (b,a)$ = $\rho (\lambda F(a),\lambda F(b))$ + $\rho (\lambda F(b),\lambda F(a))$ 
$< \rho (a,b)$ + $\rho (b,a)$, a contradiction.

{\bf Case 2}. The points $a$ and $b$ are not comparable.

Fix $c = \max \{a, b\}  \in X_n$. We put $d$ = $\lambda F(\lambda F(c))$. By construction, $a \preceq _L d$ and $b \preceq _L d$.
Hence, $c \preceq _L d$ and $\rho (a,c) \leq \rho (a,d)$,  $\rho (b,c) \leq \rho (b,d)$. 
Therefore $\rho (a,c) + \rho (c,a)$ $\leq \rho (a,d) + \rho (d(a)$.
  By virtue of the conditions $\Omega _i$, we then have  $ \rho (a,d) + \rho (d,a)$  $< \rho (a,c) + \rho (c,a)$, a contradiction.
\end{proof}
Now,  according to \cite{Meir}, consider the following two classes of Meir-Keeler type assumptions.
 
 {\bf Conditions $MK_1$}: \\
 \indent1. For any  two points $x, y \in X$ there exists an upper bound and a lower bound;
 
 2 (Meir-Keeler monotone contraction condition). There exists a function $\delta : (0, +\infty ) \longrightarrow (0,+\infty )$ such that from $r > 0$, $x, y \in X$, $d(x,y) <  r + \delta (r)$
 and $x \preceq  y $ it follows that $d(x,y)  < r$;

3. If $x, y \in X_m$, $x \preceq _Ly$, then $\lambda F(x) \preceq _L \lambda F(y)$.
 
 {\bf Conditions $MK_2$}: \\
 \indent 1. For any  two points $x, y \in X$ there exist an upper bound and a lower bound;
 
 2 (Meir-Keeler monotone contraction condition). There exists a function $\delta : (0, +\infty ) \longrightarrow (0,+\infty )$ such that
from $r > 0$, $x, y \in X$, $d(x,y) <  r + \delta (r)$ and $x \preceq  y $ it follows that $d(x,y)  < r$;

3. If $x, y \in X_m$, $x \preceq _Ly$, then $\lambda F(y) \preceq _L \lambda F(x)$.

\begin{thm}\label{T6.2}     Let $(X, d)$ be an $H$-distance space and let $a \in X^m$ be a multidimensional 
fixed point of the operator of $F: X^m \rightarrow X$.
Then in any of the Conditions $MK_i$, $i\in \{1, 2\}$, the operator $F$ has a unique multidimensional fixed point.
 \end{thm}    

\begin{proof} Obviously, in $(X^m, \preceq _L)$, for any  two points $x, y \in X^m$, there exist an upper bound and a lower bound.
 Let $\rho $ = $d^m$. In this case for any  two points $x, y \in X^m$, from the condition 
  $\rho (x,y) <  r + \delta (r)$ and $x \preceq  y $, it follows that $\rho (x,y)  < r$.
  
Assume that $b\in X^m$ is a  multidimensional fixed point of the operator of $F$ and that $b \not= a$.

{\bf Case 1}. The points $a$ and $b$ are comparable.

Assume that $a \preceq L b$ and $\rho (a,b) = r > 0$. Since $\rho (a, b) < r + \delta (r)$, we have $r$ = $\rho (a,b)$ =
 $\rho (\lambda F(a),\lambda F(b)) < r$, a contradiction.

{\bf Case 2}. The points $a$ and $b$ are not comparable.
We put $r$ = $\inf\{max\{\rho (a,c), \rho (b,c)\}: c \in X^m, a \preceq _L c, b \preceq _L c\}$. 
We claim that $ r$ = 0. Assume that $r > 0$. Then $\delta (r) > 0$ and there exists $c$ such that $\max\{\rho (a,c), \rho (b,c)\} < r + \delta (r)$,
$a \preceq _L c$, $b \preceq _L c$. We put $e$ = $\lambda F(\lambda F(c))$. Then $a \preceq _L e$ and $b \preceq _L e$. Since $\lambda F(\lambda F(a))$
= $a$ and $\lambda F(\lambda F(b))$ = $b$, we have $\max\{\rho (a,e), \rho (b,e)\} < r$, a contradiction. Thus $r$ = $0$.
For each $n \in \mathbb N$ there exists a point $c_n \in X^n$ such that  $a \preceq _L c_n$, $b \preceq _L c_n$ and
 $max\{\rho (a,c_n), \rho (b,c_n)\} < 2^{-n}$. We can construct a sequence $\{c_n: n \in \mathbb N\}$ for which
 $a$ = $\lim_{n \rightarrow \infty } c_n$ and  $b$ = $\lim_{n \rightarrow \infty } c_n$, a contradiction.
\end{proof}
In the particular case $m = 2$, the following theorems were proved in \cite{BP}. Their proofs in the general case are similar and we omit them.
 
\begin{thm}\label{T6.3}     Let $(X, d, \preceq )$ be an ordered metric space, $m \in \mathbb N$, $F: X^m \rightarrow X$ be an operator. Suppose that:

a) there exists a function $\delta : (0, +\infty ) \longrightarrow (0,+\infty )$ such that from $r > 0$, $x, y \in X^m$, $d^m(x,y) <  r + \delta (r)$
 and $x \preceq  y $ it follows that $d^m(\lambda F(x),\lambda F(y))  < r$;
 
 b)  for any  two points $x, y \in X$ there exists an upper bound and a 
lower bound.  

Suppose also that one of the following sets of conditions is satisfied:
 
 1.  $(X, d, \preceq )$ is monotonically complete; from $x, y \in X^m$ and $x \preceq _L y$ it follows that $\lambda F(x) \preceq _L \lambda F(y)$; 
 there exists $a \in X^m$ such that  $a \preceq _L \lambda F(a)$.

2.  $(X, d, \preceq )$ is complete; from $x, y \in X^m$ and $x \preceq _L y$ it follows that $\lambda F(y) \preceq L \lambda F(x)$; 
 there exists $a \in X^m$ such that  $a \preceq _L \lambda F(a)$ or $ F(a)  \preceq _L a$.
 
 Then there exists a unique multidimensional fixed point of the operator of $F$.
 \end{thm}

\begin{thm}\label{T6.4}     Let $(X, d, \preceq )$ be an ordered metric space, $m \in \mathbb N$, $F: X^m \rightarrow X$ be an operator.  Suppose that:

a) there exists a function $\delta : (0, +\infty ) \longrightarrow (0,+\infty )$ such that from $r > 0$, $x, y \in X^m$, $\bar{d}^m(x,y) <  r + \delta (r)$
 and $x \preceq  y $ it follows that $\bar{d}^m(\lambda F(x),\lambda F(y))  < r$;
 
 b) for any  two points $x, y \in X$, there exist an upper bound and a 
lower bound; for any $i \in \{1,2,\dots,m\}$, the mapping $\lambda _i$ is a surjection
or, more generally, $|\cup \{\lambda _i^{-1}(j): 1\leq j \leq m\}| = m$, for each $i \in \{1,2,\dots,m\}$.  

Suppose also that one of the following sets of conditions is satisfied:
 
 1.  $(X, d, \preceq )$ is monotonically complete; from $x, y \in X^m$ and $x \preceq _L y$ it follows that $\lambda F(x) \preceq _L \lambda F(y)$; 
 there exists $a \in X^m$ such that  $a \preceq _L \lambda F(a)$.

2.  $(X, d, \preceq )$ is complete; from $x, y \in X^m$ and $x \preceq _L y$ it follows that $\lambda F(y) \preceq L \lambda F(x)$; 
 there exists $a \in X^m$ such that  $a \preceq _L \lambda F(a)$ or $ F(a)  \preceq _L a$.
 
 Then there exists a unique multidimensional fixed point of the operator of $F$.
 \end{thm}

\section{Some particular cases and a generic application of multiple fixed points}
If we take concrete values of $m\in \mathbb{N}$ and consider various particular functions  $\lambda $ = $\{\lambda _i: \{1, ..., m\} \rightarrow  \{1, ..., m\}: 1\leq i \leq m\}$ then, most of the concepts of coupled, triple, quadruple,..., multiple fixed point theory existing in literature are  obtained as particular cases of the concept of multiple fixed point considered in \cite{ChoB} and the present paper. 

For example, if $m=2$, $\lambda_1(1)=1$, $\lambda_1(2)=2$; $\lambda_2(1)=2$, $\lambda_2(2)=1$,  we obtain  the incept of coupled fixed point studied in \cite{Bha} and in various subsequent papers. If $m=3$, $\lambda_1(1)=1$, $\lambda_1(2)=2$, $\lambda_1(3)=3$; $\lambda_2(1)=2$, $\lambda_2(2)=1$, $\lambda_2(3)=2$; $\lambda_3(1)=3$, $\lambda_3(2)=2$, $\lambda_3(3)=1$, then the concept of multiple fixed point studied in the present paper reduces to that of triple fixed point, first introduced in \cite{BB1} and intensively studied in many other research works emerging from it.  

We note that, as pointed out in \cite{Sha},  the notion of tripled fixed point due to Berinde and Borcut \cite{BB1} is different from the one defined by Samet and Vetro \cite{SV} for $N = 3$, since in the case of ordered metric spaces in order to keep the mixed monotone property working, it is necessary to take $\lambda_2(3)=2$ and not $\lambda_2(3)=3$.

It is also important to mention here that some cases of multidimensional coincidence points results (that extend multiple fixed point theorems) are not compatible with the mixed monotone property (see \cite{Al}).

For other concepts of multiple fixed points considered in literature the condition " $\lambda _i$ is a surjection, for each $i   \leq m$" is no more valid, see for example \cite{BB1} and the research papers emerging from it, while the second condition, $|\cup \{\lambda _i^{-1}(j): 1\leq j \leq m\}| = m$, for each $i   \leq m$, is satisfied.

Finally, we point out the fact that our approach in \cite{ChoB} and in this paper is based on the idea to obtain general multiple fixed point theorems by reducing this problem to a unidimensional fixed point problem and by simultaneously working in a more general and very reliable setting, i.e., that of distance space. 

Many other related and relevant results could be obtained in the same way, by reducing the multidimensional fixed point problem to many other independent unidimensional fixed point principles, like the ones established in  \cite{Ber13},  \cite{Ber09},  \cite{Ber09a},  \cite{Ber10},  \cite{Ber11a},  \cite{Ber12},  \cite{BC1},  \cite{BP},  \cite{BerP15} etc. 

We end the paper by indicating an interesting  generic application of multiple fixed points in game theory. 

Fix an orderable distance space  $(X, d, \preceq )$ and a positive integer number $m \geq  2$. We put $\mathbb N$ = 
$\{1, 2, ...\}$ and $\mathbb N_m$ = $\{1, 2, ..., m\}$. For $L \subseteq \mathbb N$, we introduce on $X$ 
the ordering $\preceq _L$: $ (x_1,..., x_m) \preceq _L (y_1,...,y_m)$ iff $x_i \preceq y_i$ for $i \in L$ 
and $y_j \preceq x_j$ for $j \notin L$.

Denote by $\Lambda  $ = $(\lambda _1,... , \lambda _m)$ a collection of mappings
$\{\lambda _i: \mathbb N_m$ $\longrightarrow \mathbb N_m: i \leq  m\}$.

Let  $F: X^m \longrightarrow X$ be an operator. The operator $F$ and the
mappings $\Lambda $ generate the operator $\Lambda F: X^m \longrightarrow X^m$, where $\lambda F(x_1,....,x_m)$
= $(y_1,...,y_m)$ and $y_i$ = $F(x_{\lambda _i(1)}, ..., x_{\lambda _i(m)})$ for each point $(x_1, ..., x_m) \in X^m$
and any index $i \leq  m$. 

Assume now that $\mathbb N_m$ is the set of players and $i \in \mathbb N_m$ is the symbol of the $i$th player.
In this case we say that:

- $X$ is the space of the positions (decisions) of the players;

- $d(x, y)$ is the measure of the non-convenience of the position $x$ relatively to the position $y$;

- ordering $\leq $ is the relation of domination of positions;

- a point $x $ = $(x_1, x_2, ..., x_m) \in X^m$ is a selection of positions, where $x_i$ is the position of the
player $i$;

-  the operator $F$ is the operator of correction of the positions. 

Every selection of positions  $x $ = $(x_1, x_2, ..., x_m) \in X^m$ determines the  selection of positions  
  $y $ = $(y_1, y_2, ..., y_m)$ = $\Lambda F(x) \in X^m$. 
  
  For any player $i$ the number $d(x_i,y_i)$ is the 
measure   of the non-convenience of the position $x_i$ relatively to the position $y_i$ for the $i^{th}$ player.

One considers that the selection of the positions  $x $ = $(x_1, x_2, ..., x_m) \in X^m$ is optimal if $d(x_i,y_i)$ is minimal for each $i$.

In particular, if $\Lambda F(x) $ = $x$, then the selection of positions  $x$ is optimal. 

One can find distinct concrete examples of the above general model in \cite{Nik}, \cite{Mal}, \cite{Op75}, \cite{Op75a}.
 
\section*{Acknowledgements}

This  second author acknowledges the support provided by the Deanship of Scientific Research at King Fahd University of Petroleum and Minerals for funding this work through the projects IN151014 and IN141047.

\vskip 0.25 cm {\it $^{a}$ Department of Physics, Mathematics and Information
Technologies\newline
\indent Tiraspol State University\newline 
\indent  Gh. Iablocikin 5., MD2069 Chi\c sin\u au, Republic of Moldova

E-mail: mmchoban@gmail.com }

\vskip 0.5 cm {\it $^{b}$ Department of Mathematics and Computer Science

North University Center at Baia Mare 

Technical University of Cluj-Napoca

Victoriei 76, 430072 Baia Mare ROMANIA

E-mail: vberinde@cunbm.utcluj.ro}

\vskip 0.25 cm {\it $^{c}$ Department of Mathematics and Statistics

King Fahd University of Petroleum and Minerals

Dhahran, Saudi Arabia

E-mail: vasile.berinde@gmail.com}
\end{document}